\documentclass[12pt]{article}
\usepackage{amsmath,amsopn,amssymb,amsthm,amsfonts}
\usepackage[T2A]{fontenc}
\usepackage[cp1251]{inputenc}
\numberwithin{equation}{section}

\newtheorem{theorem}{Theorem}[section]

\newtheorem{lemma}[theorem]{Lemma}
\newtheorem{remark}[theorem]{Remark}

\def\Ker {{\rm Ker}}
\def\Im {{\rm Im}}

\def\dim {{\rm dim}}
\def\div {{\rm div}}
\def\rot {{\rm rot}}

\def\grad {{\rm grad}}

\begin{document}

\vspace{7mm}

\centerline{\large \bf Bi-invariant metric on volume-preserving diffeomorphisms  group}
\centerline{\large \bf of a three-dimensional  manifold}

\vspace{3mm}

\centerline{\large \bf N.~K.~Smolentsev }

\vspace{7mm}

\begin{abstract}
We show the existence of a weak bi-invariant symmetric nondegenerate 2-form  on  the  volume-preserving  diffeomorphism  group  of  a  three-dimensional  manifold  and  study  its properties. Despite the fact that the space $\mathcal{D}_\mu(M^3)$ is infinite-dimensional, we succeed in defining the signature of the bi-invariant quadric form.  It is equal to the $\eta$-invariant of the manifold $M^3$.
\end{abstract}

\section{Invariant form on $d\Gamma(\Lambda^{2k-1}M)$}
Let $(M,g)$ be a smooth (of class $C^\infty$) compact Riemannian manifold
of dimension $n = 4k-1$ without boundary.  Consider the elliptic self-adjoint operator $A$ acting on smooth exterior differential forms on $M$ of even degree $2p$, given by
$$
A\varphi=(-1)^{k+p+1}(*d-d*)\varphi,\quad \deg \varphi =2p.
$$
Here $d$ is exterior differential and $*$ is the Hodge duality operator defined by metric. Let $\delta = (-1)^{np+n+1}*d*$ is the codifferential acting on exterior differential $p$-forms on $M$.

The  space $\Gamma(\Lambda^{2k}M)$ of  smooth  forms  on $M$ of  degree $2k$ is  invariant  under  the  action  of  $A$. Let us write the Hodge decomposition of the space $\Gamma(\Lambda^{2k}M)$ into a direct sum:
$$
\Gamma(\Lambda^{2k}M) = \delta\Gamma(\Lambda^{2k+1}M)\oplus
H^{2k}\oplus d\Gamma(\Lambda^{2k-1}M),
$$
where $\delta\Gamma(\Lambda^{2k+1}M)$ is  the  space  of  co-exact  forms  and $H^{2k}$ is  the  space  of  harmonic  forms, and  let $d\Gamma(\Lambda^{2k-1}M)$  be the  space  of  exact $2k$-forms.  The  kernel  of  the  operator $A=(-1)^{2k+1}(* d -d *) = -* d +d *$ on  the  space
$\Gamma(\Lambda^{2k}M)$  consists of harmonic forms and the image obviously coincides with the space of exact and co-exact forms.
Moreover, the first and second components of the operator $A=-* d +d *$ separately act on $\delta\Gamma(\Lambda^{2k+1}M)$ и $d\Gamma(\Lambda^{2k-1}M)$:
$$
\Im(-*d)=\delta\Gamma(\Lambda^{2k+1}M), \qquad
\Ker(-*d)=H^{2k}\oplus d\Gamma(\Lambda^{2k-1}M),
$$
$$
\Im(d*)=d\Gamma(\Lambda^{2k-1}M), \qquad \Ker(d *)= H^{2k} \oplus
\delta\Gamma(\Lambda^{2k+1}M).
$$
Therefore, the restriction of the operator $A=- * d+d *$ to the direct sum  $\delta\Gamma(\Lambda^{2k+1}M)\oplus d\Gamma(\Lambda^{2k-1}M)$ is an isomorphism preserving this decomposition.
Hence the operator $A$ has an inverse $A^{-1}$ on the space
$\delta\Gamma(\Lambda^{2k+1}M)$, as well as on the space $d\Gamma(\Lambda^{2k-1}M)$.  In what follows, we use the operator
$$
A^{-1} :d\Gamma(\Lambda^{2k-1}M) \rightarrow d\Gamma(\Lambda^{2k-1}M)
$$
inverse to the operator $A = d *$ on the space $d\Gamma(\Lambda^{2k-1}M)$.

On the space $\Gamma(\Lambda^{2k}M)$, there exists the natural inner product
\begin{equation}
(\alpha,\beta)= \int_M \alpha\wedge *\beta, \qquad \alpha,\beta
\in \Gamma(\Lambda^{2k}M). 
\label{(alpha,beta)-9.1}
\end{equation}
Then the Hodge decomposition $\Gamma(\Lambda^{2k}M) = \delta\Gamma(\Lambda^{2k+1}M) \oplus H^{2k}\oplus d\Gamma(\Lambda^{2k-1}M)$ is orthogonal. Let $p:\Gamma(\Lambda^{2k}M) \rightarrow d\Gamma(\Lambda^{2k-1}M)$ be the orthogonal projection.

Let $L_X=d\cdot i_X +i_X \cdot d$ be the Lie derivative, where $i_X$ is the inner product, $i_X \alpha = \alpha(X,\cdot)$.

\begin{lemma} \label{Lem-1}
On the space $d\Gamma(\Lambda^{2k-1}M)$ the following equality is fulfilled
$$
A^{-1}\cdot L_X = p\, *\, i_X
$$
for any vector field $X$.
\end{lemma}

\begin{proof}
Let  $d\alpha \in d\Gamma(\Lambda^{2k-1}M)$. Then $*i_X d\alpha =d\beta +\gamma$ where $d\beta \in d\Gamma(\Lambda^{2k-1}M)$ and $\gamma\in \Ker \delta = \Ker p = \Ker (d*)$. We will notice that on the space $d\Gamma(\Lambda^{2k-1}M)$ we have $L_X=d\cdot i_X$ and $A=d*$. Therefore,
$$
p*i_X(d\alpha) = d\beta,
$$
$$
A^{-1}\cdot L_X (d\alpha) =(d*)^{-1}d\, i_X(d\alpha)= (d*)^{-1}d**i_X(d\alpha)=
$$
$$
=(d*)^{-1}(d*)(d\beta +\gamma)=(d*)^{-1}(d*)(d\beta)=d\beta .
$$
\end{proof}

\begin{lemma} \label{Lem-2}
The operator $p*i_X$ is skew-symmetric on the space $d\Gamma(\Lambda^{2k-1}M)$:
$$
(p*i_X\, d\alpha, d\beta) +(d\alpha, p*i_X\, d\beta)=0
$$
for all $d\alpha, d\beta \in d\Gamma(\Lambda^{2k-1}M)$.
\end{lemma}

\begin{proof}
As the operator $p$ is orthogonal projection on $d\Gamma(\Lambda^{2k-1}M)$, then
$$
(p*i_X\, d\alpha, d\beta) +(d\alpha, p*i_X\, d\beta)=(*i_X\, d\alpha, d\beta) +(d\alpha, *i_X\, d\beta)=
$$
$$
=\int_M i_X\, d\alpha \wedge d\beta +\int_M  d\alpha \wedge **i_X\, d\beta=
\int_M \left( i_X\, d\alpha \wedge d\beta +d\alpha \wedge i_X\, d\beta\right)=0,
$$
as $i_X\, d\alpha \wedge d\beta +d\alpha \wedge i_X\, d\beta=i_X(d\alpha \wedge d\beta)=0$.
\end{proof}

Consider the following bilinear form $Q$ on $d\Gamma(\Lambda^{2k-1}M)$ introduced in \cite{APS1}:  for $d\alpha,d\beta \in d\Gamma(\Lambda^{2k-1}M)$,
\begin{equation}
Q(d\alpha,d\beta)=(d\alpha,A^{-1}d\beta)=\int_M d\alpha\wedge \beta. 
\label{Q(dalpha,dbeta)-9.2}
\end{equation}

Since the operator $A$ is self-adjoint, the form $Q$ is symmetric.  The signature of the corresponding quadratic form $Q(d\alpha,d\alpha)$ is equal to the $\eta$-invariant of the manifold $M$  \cite{APS1}.  If $M$ is the boundary of a $4k$-dimensional
manifold $N$, then (see \cite{APS1})
$$
\eta=\int_N L_k - {\rm sign} N,
$$
where $L_k=L_k(p_1,\dots ,p_k)$ is  the $L$-Hirzebruch  polynomial  and ${\rm
sign} N$ is  the  signature  of  the  natural quadratic form on the $2k$-cohomology space $H^{2k}(N,\mathbb{R})$.

Let $\mathcal{D}_0$ be the connected components of the identity of the smooth diffeomorphism group of the manifold $M$.  The group $\mathcal{D}_0$ acts on the space $d\Gamma(\Lambda^{2k-1}M)$ to the right:
$$
d\Gamma(\Lambda^{2k-1}M)\times \mathcal{D}_0 \rightarrow
d\Gamma(\Lambda^{2k-1}M),  \quad\ (d\alpha,\eta) \rightarrow
\eta^*(d\alpha),
$$
where $\eta^*$ is  the  codifferential  of  a  diffeomorphism $\eta \in
\mathcal{D}_0$. Expression  (\ref{Q(dalpha,dbeta)-9.2})  for  the  quadratic  form $Q$ implies the following:

\begin{theorem} \label{Inv-Q}
The  bilinear  form $Q$  on $d\Gamma(\Lambda^{2k-1}M)$ is  invariant  under  the  action  of  the  group $\mathcal{D}_0$ on $d\Gamma(\Lambda^{2k-1}M)$: for any vector field $X$ on $M$,
\begin{equation}
Q(L_Xd\alpha, d\beta)+Q(d\alpha,L_X d\beta)= 0. 
\label{Q(LXdalpha,dbeta)-9.3}
\end{equation}
\end{theorem}

\begin{proof}
$$
Q(L_Xd\alpha, d\beta)+Q(d\alpha,L_X d\beta)= (L_Xd\alpha, A^{-1}d\beta)+(d\alpha,A^{-1} L_X d\beta)=
$$
$$
= (A A^{-1}L_Xd\alpha, A^{-1}d\beta)+(d\alpha,*i_X d\beta)= (A^{-1}L_Xd\alpha, AA^{-1}d\beta)+(d\alpha,*i_X d\beta)=
$$
$$
= (*i_Xd\alpha, d\beta)+(d\alpha,*i_X d\beta).
$$
\end{proof}

\section{Bi-invariant  metric  on  the  group $\mathcal{D}_\mu(M)$}
Let $\mathcal{D}_\mu$ be the group of diffeomorphisms of manifold $M$ leaving the Riemannian volume element $\mu$ fixed,
$$
\mathcal{D}_\mu=\{\eta \in \mathcal{D};\ \eta^*\mu =\mu \}
$$
where $\eta^*$ is the codifferential of the diffeomorphism $\eta$. If $X$ is a vector field on $M$, then the divergence $\div X$ of the field $X$  is defined by $L_X \mu = (\div X)\mu$, where $L_X$ is the Lie derivative. If $\eta_t$ is a one-parametric subgroup of the group $\mathcal{D}_\mu$,  then $\eta_t^*\mu = \mu$. Differentiating  this  relation  in $t$, we  obtain $L_X \mu =0$ or \ $\div X = 0$, where $X$ is the vector field of velocities of the flow $\eta_t$ on $M$.  Therefore, the Lie algebra of the group $\mathcal{D}_\mu$ consists of divergence-free vector fields, i.e., those fields $X$ for which $\div_\mu X = 0$.
Ebin  and  Marsden  showed  in \cite{Eb-Mar} that  the  group $\mathcal{D}_\mu$ is a closed  ILH-subgroup  of  the  ILH-Lie  group $\mathcal{D}$ with the Lie algebra $T_e \mathcal{D}_\mu$ consisting of divergence-free vector fields on $M$. The formula
\begin{equation}
(X,Y) = \int_M g(X,Y) d\mu, \quad X,Y\in T_e \mathcal{D}_\mu
\label{(X,Y)-Dmu}
\end{equation}
is defined on $\mathcal{D}_\mu$ a smooth right-invariant weak Riemannian structure.
In  \cite{Arn} and \cite{Eb-Mar} it  was  shown  that  geodesics  on  the  group $\mathcal{D}_\mu$ are flows  of  the  ideal incompressible fluid. (At integration the form $\mu$ we will note as $d\mu$).

\vspace{2mm}
On a Riemannian manifold, there exists a natural isomorphism between the space $\Gamma(TM)$ of smooth vector fields on $M$ and the space $\Gamma(\Lambda^1M)$ of smooth 1-forms. To  each  vector field $Y$, it  puts  in  correspondence  the  1-form $\omega_Y$  such  that $\omega_Y(X) = g(Y,X)$, where $g$ is the Riemannian metric on $M$. If $\dim M =3$, then there exists one more isomorphism between $\Gamma(TM)$ and the space of 2-forms on $M$.  To a vector field $X$ on $M$, it puts in correspondence the 2-form $i_X\mu$, where $\mu$ is the Riemannian volume element on $M$ and $i_X$ is the inner product.  Then the sequence
of operators $\grad, \rot, \div$,
$$
0\rightarrow C^\infty(M,\mathbb{R}) \rightarrow \Gamma(TM)
\rightarrow \Gamma(TM) \rightarrow
C^\infty(M,\mathbb{R})\rightarrow 0.
$$
corresponds to the sequence of exterior differentials
$$
0\rightarrow C^\infty(M,\mathbb{R}) \rightarrow
\Gamma(\Lambda^1M) \rightarrow \Gamma(\Lambda^2M) \rightarrow
\Gamma(\Lambda^3M) \rightarrow 0
$$
The operator $\rot$ is defied by $d\omega_X=i_{\rot X}\mu$.

We see from the relation $di_X\mu = L_X \mu =(\div X)\mu$ that the 2-form $i_X\mu$ is closed iff $\div X = 0$.  Clearly, the 2-form $i_X\mu$ is exact iff $X =
\rot V$ and $i_X\mu = d\omega_V  = i_{\rot V}\mu$.  The space $d\Gamma(\Lambda^1M)$ is isomorphic to the subspace $\Im(\rot)\subset \Gamma(TM)$.  Note that
$$
(\omega_Y,\omega_X)=(i_Y\mu,i_X\mu)=(Y,X)= \int_M g(Y,X)d\mu
$$
and  $*\omega_Y=i_Y\mu$, $*i_X \mu =\omega_X$. This immediately  implies  that  the  operator $\rot:\Im(\rot) \rightarrow \Im(\rot)$ corresponds to the operator $A = d*:d\Gamma(\Lambda^1M) \rightarrow d\Gamma(\Lambda^1M)$. Indeed, each exact 2-form $d\alpha \in d\Gamma(\Lambda^1M)$ is represented in the form $d\alpha =
i_X\mu$, where $X \in \Im(\rot)$; therefore,
$$
A(i_X\mu)=(d*)(i_X\mu)=d(\omega_X)=i_{\rot X}\mu.
$$
The operator $A^{-1}$ corresponds to the operator $\rot^{-1}$, $A^{-1}(i_X\mu) = i_{\rot^{-1}X}\mu$.

We have represented every 2-form $d\alpha \in d\Gamma(\Lambda^1M)$ in the form $d\alpha =i_X\mu$, $X\in \Im (\rot)$.  Let us find the corresponding expression of the symmetric 2-form $Q$ on the space $d\Gamma(\Lambda^1M)$:
$$
Q(d\alpha,d\beta) = Q(i_X\mu,i_Y\mu) = (i_X\mu,A^{-1}i_Y\mu)= (i_X\mu,i_{\rot^{-1}Y}\mu)=
$$
$$
= (X,\rot^{-1}Y)=\int_M g(X,\rot^{-1}Y) d\mu.
$$

Therefore, to the bilinear form $Q$ on $d\Gamma(\Lambda^1M)$, we put in correspondence the following bilinear symmetric form on $\Im(\rot)$:
\begin{equation}
\left<X,Y\right>_e =Q(i_X\mu,i_Y\mu)= \int_M g(X,\rot^{-1}Y) d\mu, 
\label{<X,Y>-9.4}
\end{equation}
where $X,Y \in \Im(\rot)\subset \Gamma(TM)$. This form  is  nondegenerate;  indeed,  for any $X \in \Im(\rot)$, we  have $\left<X,\rot X\right> = (X,X)> 0$ if $X\neq 0$. The symmetry of the form (\ref{<X,Y>-9.4}) follows from the self-adjointness of the operator $\rot :\Gamma(TM) \rightarrow \Gamma(TM)$.

The  space $\Im (\rot)$ is  a  Lie  subalgebra  of  the  Lie algebra $T_e
\mathcal{D}_\mu$, and,  moreover,  it  is  its  ideal. This follows from the fact that the Lie bracket $[X,Y]$ of divergence-free vector fields defines the exact 2-form
$i_{[X,Y]}\mu=d(i_Xi_Y\mu)$. The  algebra $T_e \mathcal{D}_\mu$ of  divergence-free  vector fields on $M$ consists of all vector fields $X$ on $M$ for which the $(n-1)$-form $i_X\mu$ is closed: $L_X\mu = d(i_X\mu) = 0$. The algebra $\Im (\rot)$ consists of all  vector fields $X$ for  which  the $(n-1)$-form $i_X\mu$ is  exact.   Such  vector fields $X$ are  said  to  be  exact divergence-free.

In \cite{Omo5}, it was shown that there exists an ILH-Lie group $\mathcal{D}_{\mu\partial}\subset \mathcal{D}_\mu$ whose algebra Lie is the algebra
of exact divergence-free vector fields on $M$. The group $\mathcal{D}_{\mu\partial}$ is called the group of exact diffeomorphisms preserving  the  volume  element $\mu$.     Taking  this  fact  into  account,  in  what  follows,  we  denote  the  space
$\Im (\rot)$  by $T_e\mathcal{D}_{\mu\partial}$ and  consider  it  as  the  Lie  algebra  of  the  group $\mathcal{D}_{\mu\partial}$.
The  algebra $T_e\mathcal{D}_{\mu\partial}$ differs  from the algebra $T_e\mathcal{D}_{\mu}$ by the cohomology space $H^2(M,\mathbb{R})$:
$$
T_e\mathcal{D}_{\mu} =T_e\mathcal{D}_{\mu\partial}\oplus H^2(M,\mathbb{R}).
$$
If $H^2(M,\mathbb{R})=0$, then the groups $\mathcal{D}_{\mu}$ and $\mathcal{D}_{\mu\partial}$ and their Lie algebras coincide.  Therefore, the form
$$
\left< X,Y \right>_e= \int_M g(X,\rot^{-1}Y) d\mu(x)
$$
is  a  bilinear  symmetric  nondegenerate  form  on  the  Lie  algebra $T_e\mathcal{D}_{\mu\partial}$ of  the  group $\mathcal{D}_{\mu\partial}$ of  the  exact diffeomorphisms preserving the volume element $\mu$.

Using right translations, the symmetric 2-form (\ref{<X,Y>-9.4}) defines the following right-invariant symmetric 2-form on the whole group $\mathcal{D}_{\mu\partial}$: for $X_\eta ,Y_\eta \in T_\eta\mathcal{D}_{\mu\partial}$,
$$
\left<X_\eta,Y_\eta \right>_\eta= \left<dR_\eta^{-1}X,dR_\eta^{-1}Y_\eta\right>_e, \quad \eta \in \mathcal{D}_{\mu\partial}.
$$

The following theorem states that the obtained form on $\mathcal{D}_{\mu\partial}$ is smooth and bi-invariant.

\begin{theorem} [\cite{Smo8}] \label{Th9.1}
The bilinear form (\ref{<X,Y>-9.4}) on the Lie algebra $T_e\mathcal{D}_{\mu\partial}$ of the group $\mathcal{D}_{\mu\partial}$ defines the ILH-smooth  bi-invariant  form  on  the  ILH-Lie  group $\mathcal{D}_{\mu\partial}$. In  particular,  the  following  relation  holds  for any $X,Y,Z\in T_e\mathcal{D}_{\mu\partial}$:
\begin{equation}
\left<[X,Y],Z\right>_e = -\left<Y,[X,Y]\right>_e . 
\label{<[X,Y],Z>-9.5}
\end{equation}
The signature of the quadratic form corresponding to (\ref{<X,Y>-9.4}) is equal to the $\eta$-invariant of the manifold $M$.
\end{theorem}

\begin{proof}
To prove the ILH-smoothness of the form on $\mathcal{D}_{\mu\partial}$ obtained from (\ref{<X,Y>-9.4}) by right translations, we use the Omori result \cite{Omo3} on  the  smoothness  of  the  right-invariant  morphism $T\mathcal{D}^{s} \rightarrow \gamma^{s,s-1}$, $s \geq n+5+r$,  defined by  the  kernel  of  a  differential operator  of  order $r$ with  smooth  coefficients.

The  bi-invariance  property $\left<Ad_\eta X,Ad_\eta Y\right>_e = \left<X,Y\right>_e$, $\eta \in \mathcal{D}_{\mu\partial}$, of the form $\left< X,Y \right>_e= Q(i_X\mu,i_Y\mu)$ follows from the $\mathcal{D}$-invariance of the form $Q$ and the fact that $\eta^*(i_X\mu)=-i_{Ad_\eta X}\mu$, $\eta\in \mathcal{D}_\mu$.

As the group $\mathcal{D}_{\mu\partial}$ is connected, then bi-invariance property $\left<Ad_\eta X,Ad_\eta Y\right>_e = \left<X,Y\right>_e$, $\eta \in \mathcal{D}_{\mu\partial}$, of the form
(\ref{<X,Y>-9.4}) follows from (\ref{<[X,Y],Z>-9.5}). From $\mathcal{D}$-invariance of the form $Q$ on $d\Gamma(\Lambda^1M)$ we have:
$$
Q(L_X(i_Y\mu),i_Z\mu)+Q(i_Y\mu,L_X(i_Z\mu))=0.
$$
From $L_X\mu=0$ and $[L_X,i_Y]=i_{[X,Y]}$ we have:
$$
L_X(i_Y\mu) = L_X i_Y\mu - i_Y L_X\mu= i_{[X,Y]}\mu.
$$
Then, from $\left< X,Y \right>_e= Q(i_X\mu,i_Y\mu)$ we have:
\begin{multline*}
0= Q(L_X(i_Y\mu),i_Z\mu)+Q(i_Y\mu,L_X(i_Z\mu)) =Q(i_{[X,Y]}\mu,i_Z\mu)+Q(i_Y\mu,i_{[X,Z]}\mu))= {} \\* =\left<[X,Y],Z\right>_e +\left<Y,[X,Y]\right>_e.
\end{multline*}

The  operator $\rot$ is  not  positive-definite: its  eigenvalues $\lambda_i$ can  be  positive  as  well  as  negative  (the squares  of  these  eigenvalues  are  the  eigenvalues  of  the  Laplace  operator $\triangle = \rot\circ \rot$).         Therefore, the quadratic form $\left<X,X\right>_e = (X,\rot^{-1}X)_e$ is not positive-definite. The signature of this quadratic form that
is understood as the limit of the function
$$
\eta(s) = \sum_{i}({\rm sign} \lambda_i)|\lambda_i|^{-s}
$$
at  zero  is finite  and  is  equal  to  the $\eta$-invariant of the manifold $M$.    This  follows  from  the  fact  that  the $\eta$-invariant is equal to the signature of the form $Q$ on $d\Gamma(\Lambda^1M)$ \cite{APS1}.
\end{proof}

\begin{remark}  \label{Rem9.1}
If $H^2(M,\mathbb{R})=0$, then $\mathcal{D}_\mu=\mathcal{D}_{\mu\partial}$. In this case, form (\ref{<X,Y>-9.4}) is a bilinear symmetric form on the algebra of divergence-free vectors on $M$ is bi-invariant with respect to the group $\mathcal{D}_\mu(M)$ of diffeomorphisms of $M$ preserving the volume element $\mu$.
\end{remark}

\begin{remark} \label{Rem9.2}
If $\partial M \neq 0$, then on the space $T_e\mathcal{D}_{\mu\partial}(M)$ of exact divergence-free vector fields on $M$ tangent to the boundary, we can define the invariant form by
\begin{equation}
\left<X,Y\right>_e=\frac 12\left((\rot^{-1}X,Y)_e+(X,\rot^{-1}Y)_e\right), 
\label{<X,Y>-9.6}
\end{equation}
where $\rot^{-1}X$ is the vector field on $M$ tangent to $\partial
M$ such that its vorticity is equal to $X$.  If $H^2(M,\mathbb{R})=0$, then the operator r$\rot^{-1}$ is defined on $T_e \mathcal{D}_\mu$. By the usual calculations, we prove the $\mathcal{D}_\mu$-invariance of the inner product (\ref{<X,Y>-9.6}) on $T_e \mathcal{D}_\mu$.
\end{remark}

\begin{remark} \label{Rem9.3}
The  obtained  expression (\ref{<X,Y>-9.4})  for  the  bi-invariant  form  is  explained  by  the  fact  that  the group  exponential  mapping  of  the  diffeomorphism  group  is  not  surjective. Indeed, the new pseudo-Riemannian metric $\left<X,Y\right>_e
=(X,\rot^{-1}Y)_e$ is expressed through the Riemannian metric by using the
operator $\rot^{-1}$. Therefore, in the geodesic equations $\ddot{\eta}=
-\Gamma(\eta, \dot{\eta})$, we have the compact operator $\rot^{-1}$
with due account for which the exponential mapping is also compact, and hence it is not surjective.
\end{remark}

\section{Euler equations on the Lie algebra $T_e \mathcal{D}_\mu$}
Let $\mathfrak{g}$ be a semisimple, finite-dimensional Lie algebra,
and  let $H$ be  a  certain  function  on $\mathfrak{g}$. In \cite{Mi-Fo},  it  was  shown  that  the  extension  of  the  Euler  equation on  the  Lie  algebra  of  the  group $SO(n,\mathbb{R})$ of  motions  of  an $n$-dimensional  rigid  body  to  the  case  of  the general semisimple Lie algebra $\mathfrak{g}$ is an equation of the form
\begin{equation}
\frac{d}{dt}X = [X,\grad H(X)], 
\label{Eq-9.7}
\end{equation}
where $X \in\mathfrak{g}$ and the gradient of the Hamiltonian function $H$ is calculated with respect to the invariant Killing--Cartan inner product on $\mathfrak{g}$.

Assume  that  the  second  cohomology  group  of  the  manifold  $M^3$ is  trivial: $H^2(M,\mathbb{R})=0$. As $\mathfrak{g}$, let us  consider  the  Lie  algebra $T_e\mathcal{D}_\mu$ of  divergence-free  vector fields  on  the  three-dimensional  Riemannian manifold $M$. On $T_e\mathcal{D}_\mu$, we have the invariant nondegenerate form (\ref{<X,Y>-9.4}) and the function (kinetic energy)
$$
T(V)= \frac 12(V,V)_e = \frac 12 \int_M g(V,V)d\mu,\quad V \in T_e\mathcal{D}_\mu(M).
$$
The function $T$ can be written as follows in terms of the inner product (\ref{<X,Y>-9.4}):
$$
T(V) = \frac 12(V,V)_e= \frac 12(V, \rot^{-1}\rot V)_e=\frac 12\left<V,\rot V \right>_e.
$$
Perform  the  Legendre  transform $X = \rot V$;  then $T(V)
=T(\rot^{-1}X) =\frac 12\left< \rot^{-1}X,X \right>_e$. Consider  the
Hamiltonian function $H(X) =\frac 12 \left<\rot^{-1}X,X \right>_e$ on the Lie algebra $T_e\mathcal{D}_\mu$. The gradient of the function $H$ with respect to the invariant inner product (\ref{<X,Y>-9.4}) is easily calculated:
$$
\grad H(X) =  \rot^{-1} X.
$$
As in the finite-dimensional case, let us write the Euler equation on the Lie algebra $T_e\mathcal{D}_\mu(M)$:
\begin{equation}
\frac{d}{dt}X =[X,\rot^{-1}X]. 
\label{Eq-9.8}
\end{equation}
Since $X =\rot V$, it follows that in the variables $V$, this equation yields the Helmholtz equation \cite{Serrin}
\begin{equation}
\frac {\partial \rot V}{\partial t}=[\rot V,V] , 
\label{Eq-9.9}
\end{equation}
in our case, $H^2(M,\mathbb{R})=0$, it is equivalent to the equation
$$
\frac {\partial V}{\partial t}=\nabla_VV  - \grad p
$$
of motion of the ideal incompressible fluid in $M$.

On the Lie algebra $T_e\mathcal{D}_\mu$, Eq.  (\ref{Eq-9.8}) or (\ref{Eq-9.9}) has the following two quadratic first integrals:
$$
m(X)=\left<X,X\right>_e=(\rot V,V)=\int_M g(\rot V,V) d\mu,
$$
$$
H(X) = T(V) =\frac 12\int_M g(V,V) d\mu.
$$
The first of them $m(X)$ is naturally called the \emph{kinetic moment}.  The invariance of the function $m$ follows from  the  invariance  of  the  inner  product  (\ref{<X,Y>-9.4})  on $T_e\mathcal{D}_\mu(M)$. The  second  integral $H(X)$ is  the \emph{kinetic energy}.

Since $H(X)=\frac 12\left<X,\rot X\right>_e$, the operator $\rot :T_e\mathcal{D}_\mu \rightarrow T_e\mathcal{D}_\mu$ is the inertia operator of our mechanical
system $(T_e\mathcal{D}_\mu, H)$. The eigenvectors $X_i$ of the operator $\rot$ are naturally called (analogously to the rigid body  motion)  the \emph{axes  of  inertia},  and  the  eigenvalues $\lambda_i$ of  $\rot$  are  called  the  \emph{moments  of  inertia}  with respect to the axes $X_i$.

The Euler equation (\ref{Eq-9.9}) can be written in the form
$$
\frac {\partial \rot V}{\partial t}=-L_V\rot V,
$$
where $L_V$ is the Lie derivative. Therefore \cite{Ko-No}, the vector field $\rot V(t)$ is transported by the flow $\eta_t$ of the field $V(t)$:
$$
\rot V(t)=d\eta_t(\rot V(0)\circ \eta_t^{-1})=Ad_{\eta_t}(\rot V(0)),
$$
where $V(0)$ is the initial velocity field. In mechanics, this property is called the property of the vorticity $X =\rot V$ to be carried along the fluid flow $\eta_t$ \cite{Arn-Kh}.  Hence the curve $X(t)= \rot V(t)$ on the Lie algebra $T_e\mathcal{D}_\mu$ that is a solution of the Euler equation (\ref{Eq-9.8}) lies on the orbit $\mathcal{O}(X_0)=\{Ad_\eta X(0);\ \eta \in \mathcal{D}_\mu \}$ of the
coadjoint action of the group $\mathcal{D}_\mu$.

Since the kinetic moment $m(X)$ is preserved under the motion, the orbit $\mathcal{O}(X_0)$ lies on the ''pseudo-sphere''
$$
S =\{Y \in T_e\mathcal{D}_\mu;\  \left<Y,Y\right>_e = r\},
$$
where $r=\left<X(0),X(0)\right>_e$. It  is  natural  to  expect  that  the  critical  points  of  the  function $H(X)$ on  the orbit $\mathcal{O}(X_0)$ are stationary motions.  The orbit $\mathcal{O}(X_0)$ is the image of the smooth mapping
$$
\mathcal{D}_\mu \rightarrow T_e\mathcal{D}_\mu,\quad  \eta \rightarrow Ad_\eta X_0.
$$
Therefore, the tangent space $T_Y\mathcal{O}(X_0)$ to the orbit at a point $Y$ is
$$
T_Y\mathcal{O}(X_0)=\{[W,Y];\  W \in T_e \mathcal{D}_\mu\}.
$$
A point $Y\in \mathcal{O}(X_0)$ is critical for the function $H(X)$ if $\grad H(X)$ is orthogonal to the space $T_Y\mathcal{O}(X_0)$.
Since $\rot^{-1}Y=\grad H(X)$, the latter condition is equivalent to
$$
\left<\rot^{-1}Y,[W,Y]\right>_e=0\ \mbox{ for any }W\in T_e\mathcal{D}_\mu(M).
$$
The invariance of $\left< .,. \right>_e$ implies
$$
\left<[Y,\rot^{-1}Y],W\right>_e =0\ \mbox{ для любого }W\in T_e\mathcal{D}_\mu(M).
$$
From the nondegeneracy of the form $\left< .,. \right>_e$ on $T_e\mathcal{D}_\mu(M)$, we obtain that a point $Y$ on the orbit $\mathcal{O}(X_0)$ is a critical point of the function $H$ iff
$$
[Y,\rot^{-1}Y] = 0.
$$
It follows from the Euler equation $\frac {d}{dt}Y = [Y,\rot^{-1}Y]$ that \emph{a field $Y$     is an equilibrium state of our system (i.e., a stationary motion) iff\, $Y$ is a critical point of the kinetic energy $H(X)$  on the orbit $\mathcal{O}(X_0)$}.

This  fact  was  obtained  in  Arnold's  work \cite{Arn_1}, in  which  he  also  found  the  expression  for  the  second differential of the function $H(X)$ on the orbit $\mathcal{O}(X_0)$.

A  divergence-free  vector field $X$ on a three-dimensional  manifold $M$ is called  a  \emph{Beltrami  field}  if  it  is an  eigenvector  of  the  vorticity  operator: $\rot X=\lambda X$, $\lambda \in \mathbb{R}$.  The  Reeb  field $\xi$ on  a  contact  manifold is  an  example  of  Beltrami  field, $\rot \xi=\xi$.  Beltrami fields  have  a  number  of  remarkable  properties.  In particular, a Beltrami field is the velocity field of the stationary motion of an ideal incompressible fluid;
the  Beltrami  field $X$ is  a  critical  point  of  the  kinetic  energy $H(X)$ among  all  the  fields  obtained  from $X$  by  the  action  of  the  diffeomorphism  group.         The  planes  orthogonal  to  the  Beltrami field defines  a contact structure. By a  \emph{Beltrami field} one also means a divergence-free field $X$  parallel to its vorticity: $\rot X  = fX$, where $f$ is a certain function on $M$. The topology and the hydrodynamics of the Beltrami fields  were  studied  in \cite{EG}-\cite{EG3}. The  topology  of  stationary  flows  for  which  the  vorticity  vector $\rot V$ is
noncollinear to the velocity field $V$ almost everywhere was studied in \cite{Arn}, \cite{Arn_3}, \cite{Arn-Kh}, \cite{Arn_1}.

\section{Curvature of the group $\mathcal{D}_\mu(M^3)$}
For the weak bi-invariant pseudo-Riemannian structure (\ref{<X,Y>-9.4})
on  the  group $\mathcal{D}_\mu(M^3)$,  the  covariant  derivative,  the  curvature  tensor,  and  the  sectional  curvatures  of  the bi-invariant metric have the usual form:
\begin{equation}
\nabla^0_X Y = \frac 12 [X,Y],    
\label{nabla0-XY-9.10}
\end{equation}
\begin{equation}
R^0(X,Y)Z = - \frac 14 [[X,Y],Z],  
\label{R0(X,Y)Z-9.11}
\end{equation}
\begin{equation}
K_\sigma^0 = \frac 14 \left<[X,Y],[X,Y]\right>_e .
\label{K-sigma0-9.12}
\end{equation}
Taking  into  account  the  fact  that $[X,Y]=\rot(Y\times X)$, where $\times $ is  the  vector  product  on  a  three-dimensional  Riemannian  manifold,  we  obtain  the  following  formula  for  the  sectional  curvatures  of  the
bi-invariant  metric  (\ref{<X,Y>-9.4}): if $\sigma$ is  a  plane given  by  an  orthonormal  (with  respect  to  (\ref{<X,Y>-9.4}))  pair  of vectors $X,Y \in T_e\mathcal{D}_\mu(M)$, then
$$
K_\sigma^0 = \frac 14 \int_M g([X,Y],Y\times X ) d\mu(x).
$$
To find the sectional curvatures of the group $\mathcal{D}_{\mu}(M^3)$ with respect to the right-invariant weak Riemannian structure,  we  apply  the  general  formula  of  the  previous  section. In  our  case,  it  can  be  simplified. On a three-dimensional  manifold $M$, the  following  elementary  formulas  hold  for  any  vector fields $X$  and $Y$ on $M$:
$$
X\times \rot Y + Y\times \rot X = -(\nabla_XY + \nabla_YX)+\grad(X,Y),
$$
$$
\nabla_XX =\rot X\times X + \frac 12 \grad(X,X).
$$
Therefore, the projector $P=\rot^{-1}\circ\rot$ of the space $\Gamma(TM)$ of vector fields on $M$ on the space $T_e\mathcal{D}_\mu(M)$ of exact divergence-free fields on $M$  acts as follows:
$$
P\left(\nabla_XY +\nabla_YX\right)=\rot^{-1}([X,\rot Y]+[Y,\rot X]),
$$
$$
P(\nabla_XX)=\rot^{-1}[X,\rot X].
$$

\begin{theorem} [\cite{Smo8}] \label{Th9.2}
The sectional curvature of the group $\mathcal{D}_\mu(M)$ with respect to the right-invariant weak  Riemannian  structure (\ref{<X,Y>-9.4}) in  the  direction  of  a plane $\sigma$ given  by  an  orthonormal  pair  of  vectors $X,Y \in T_e\mathcal{D}_\mu(M^3)$ is expressed by the formula
\begin{multline}
K_\sigma = -\frac 12\int_M g\left(X,[[X,Y],Y]\right)d\mu -
\frac 12\int_M g\left([X,[X,Y]],Y\right) d\mu - {} \\*
-\frac 34\int_M g\left([X,Y],[X,Y]\right) d\mu +
\int_M g\left(\rot^{-1}[X,\rot X],Y\times \rot Y\right) d\mu - {} \\*
-\frac 14\int_M g\left(\rot^{-1}([X,\rot Y]-[\rot X,Y]),X\times
\rot Y - \rot X\times Y \right) d\mu.  
\label{K-sigma-9.13}
\end{multline}
\end{theorem}

If the vector fields $X$ and $Y$ are eigenvectors of the operator $\rot$, $\rot X =\lambda X$, $\rot Y =\mu Y$, then formula (\ref{K-sigma-9.13}) takes a simpler form:
\begin{multline}
K_\sigma = -\frac 12\int_M g\left(X,[[X,Y],Y]\right)d\mu -
\frac 12\int_M g\left([X,[X,Y]],Y\right)d\mu - {} \\*
-\frac 34\int_M g\left([X,Y],[X,Y]\right) d\mu - \frac
{(\lambda-\mu)^2}{4}\int_M g\left(\rot^{-1}[X,Y],X\times Y\right) d\mu. 
\label{K-sigma-9.14}
\end{multline}

\end{document}